\documentclass[a4paper, 10pt]{article}

\usepackage{packageList}	%
\usepackage{macros}

\graphicspath{{./Figures/}}
\pgfplotsset{compat=1.11}
\newlength\fwidth

\title{Spectral alignment of kernel matrices and applications}
\author[1]{Tizian Wenzel\thanks{tizian.wenzel@uni-hamburg.de}, Armin Iske}

\usepackage{macros}

\newif\iflong			
\longfalse

\usepackage[normalem]{ulem}

\begin{document}

\maketitle %
  
\begin{abstract}
Kernel matrices are a key quantity in kernel-based approximation, 
and important properties such as stability and algorithmic convergence can be analyzed with their help.

In this work we refine a multivariate Ingham-type theorem, which is then leveraged to obtain novel and refined stability estimates on kernel matrices.
For this, we focus on the case of finitely smooth kernels, 
such as the family of Matérn or Wendland kernels, 
while noting that the results also extend to norm-equivalent kernels.
In particular we obtain results that relate the Rayleigh quotients of kernel matrices for kernels of different smoothness to each other.
Finally we comment on conclusions for the eigenvectors of these kernel matrices.
\end{abstract}

\section{Introduction}

Kernel based approximation is a popular tool in different areas of applied mathematics, 
such as radial basis function approximation \cite{buhmann2000radial}, 
scattered data approximation \cite{wendland2005scattered} or recently also in the analysis of neural networks \cite{jacot2018neural}.
These kernels are symmetric functions $k: \Omega \times \Omega \rightarrow \R$ defined on some set $\Omega$,
for which frequently some positive definiteness properties are assumed.
In this paper we focus on kernels defined on $\Omega \subseteq \R^d$.
Given points $X \subset \Omega$, the kernel matrix is defined via all pairwise kernel evaluations, i.e.\ $(k(x_i, x_j))_{x_i, x_j \in X} \in \R^{|X| \times |X|}$.
With slight abuse of notation, we sometimes use the notation $k(X, X)$ for the kernel matrix.
If this kernel matrix is positive definite for any choice of pairwise distinct points $X \subset \Omega$, 
then the kernel $k$ is denoted as a strictly positive definite kernel.

This kernel matrix is fundamental for the implementation and analysis of kernel based algorithms.
For example, greedy kernel algorithms essentially realize a decomposition of the kernel matrix \cite{santin2017convergence},
and the convergence analysis relies on properties of the kernel matrix \cite{wenzel2021novel,wenzel2023analysis}.
Furthermore, properties of the kernel matrix can be related to properties of the Mercer expansion of kernels  \cite{steinwart2012mercer,santin2016approximation},
thus connecting to the field of approximation theory.
More generally, estimates on the spectrum of the kernel matrix -- and in particular on its smallest eigenvalue -- are fundamental for a wide range of applications. 
For instance, they directly lead to results in bandlimited interpolation \cite[Section 3]{narcowich2006sobolev}, 
to Bernstein inequalities for kernel spaces \cite[Lemma 3.2]{cheung2018convergence} and enable showing inverse statements for kernel based approximation \cite{schaback2002inverse,wenzel2025sharp}.
Additionally, due to the connection with the Neural Tangent Kernel (NTK) \cite{jacot2018neural}, 
properties of kernel matrices are also linked to learning properties of neural networks, influencing factors such as convergence rates in terms of the smallest eigenvalue, spectral bias, and overall generalization performance \cite{bietti2019inductive,cao2019towards}.
Significant research has focused on analyzing the smallest eigenvalue of kernel matrices \cite{narcowich1994condition,schaback1995error,diederichs2019improved}, 
as it is key to understanding the stability of kernel-based algorithms, particularly in relation to the condition number of the kernel matrix.

In this work we advance in this line of research.
This paper is structured as follows.
In \Cref{sec:background} we recall required background information and set the notation. 
\Cref{sec:symbol_funcs} presents a refinement of a multivariate Ingham-type theorem and presents a localization estimate for typically occuring integrals revolving around the symbol function.
Subsequently, \Cref{sec:novel_refined_estimates} presents applications of these result,
namely novel stability estimates that connect kernel matrices for kernels of different smoothness.
\Cref{sec:numerical_experiments} provides numerical experiments highlighting the theoretical results.
Finally, \Cref{sec:conclusion_outlook} summarizes the paper and provides an outlook.

\section{Background}
\label{sec:background}

\subsection{Translational invariant kernels}

We will mostly deal with the frequent case of translational invariant kernels that can be written as $k(x, z) = \Phi(x-z)$ for some function $\Phi: \R^d \rightarrow \R$.
The smoothness of such kernels is frequently characterized in terms of the decay rate of the Fourier transform $\hat{\Phi}$ of the function $\Phi$.
In particular we will be interested in kernels for which the Fourier transform $\hat{\Phi}$ satisfies
\begin{align}
\label{eq:fourier_decay}
c_\Phi (1+\Vert \omega \Vert^2)^{-\tau} \leq \hat{\Phi}(\omega) \leq C_\Phi (1 + \Vert \omega \Vert^2)^{-\tau} \qquad \forall \omega \in \R^d
\end{align}
for some constants $c_\Phi, C_\Phi > 0$ and a decay rate $\tau > d/2$,
where $\Vert \cdot \Vert$ refers to the Euclidean norm of $\R^d$.
We will also use the notation $\hat{\Phi}(\omega) \asymp (1+\Vert \omega \Vert^2)^{-\tau}$ for such a two sided inequality as in Eq.~\eqref{eq:fourier_decay}.
A subclass of such translational invariant kernels is given by radial basis function (RBF) kernels,
which can be written as $k(x, z) = \Phi(x - z) = \varphi(\Vert x -z \Vert)$ for a univariate radial basis function $\varphi: \R_{\geq 0} \rightarrow \R$.
 
For every positive definite kernel, there is a unique so called reproducing kernel Hilbert space (RKHS) of functions,
where the kernel acts as a reproducing kernel.
The popular class of Matérn kernels (see \Cref{tab:overview_matern} or \cite[Section 4.4]{fasshauer2007meshfree}) %
are of the form in Eq.~\eqref{eq:fourier_decay}, 
and the reproducing kernel Hilbert spaces of these kernels can be shown to be norm equivalent to the Sobolev spaces $H^\tau(\Omega)$ under mild assumptions on the domain $\Omega \subset \R^d$.
For such kernels, much effort was done to lower bound the smallest eigenvalue $\lambda_{\min}(A_X)$ of the kernel matrix $A_X$ for pairwise distinct points $X \subset \Omega$.
For this, the separation distance $q_{X}$ was found to be important, 
which describes the minimal distance between two points of $X$ and is defined as
\begin{align}
\label{eq:sep_distance}
q_X := \min_{x_i \neq x_j \in X} \Vert x_i - x_j \Vert. %
\end{align}

We recall a typical statement that lower bounds the smallest eigenvalue in terms of the separation distance, 
see e.g.~\cite{schaback1995error}: %

\begin{theorem}
\label{th:stability_statement}
Let $k$ be an RBF kernel such that Eq.~\eqref{eq:fourier_decay} holds for some $\tau > d/2$
and let $\Omega \subset \R^d$ be bounded. 

Then there is a $c_{\min} > 0$ (only depending on $\tau, d$, $\Omega$, $\Phi$) such that for any set of pairwise distinct points $X \subset \Omega$ the smallest eigenvalue of the kernel matrix $A_X = k(X, X)$ can be bounded as
\begin{align}
\label{eq:lambda_min_sobolev1}
\lambda_{\min}(A_{X}) \geq c_{\min} q_{X}^{2\tau - d}
\end{align}
and thus %
\begin{align*}
\left \Vert A_X^{-1} \right \Vert_{2,2} \leq c_{\min}^{-1} q_{X}^{d - 2\tau}.
\end{align*}
\end{theorem}

The proof of \Cref{th:stability_statement} is based on Fourier techniques,
which have become the main tool for such bounds \cite{narcowich1994condition,schaback1995error,diederichs2019improved}.
The idea is to write the Rayleigh quotient of the kernel matrix with help of the Fourier transform of the translational invariant kernel via
\begin{align}
\label{eq:kernel_vector_product}
\alpha^\top A_{X} \alpha &= (2\pi)^{-d/2} \cdot \int_{\R^d} \hat{\Phi}(\omega) \cdot \left \vert \sum_{j=1}^{|X|} \alpha_j e^{i \omega^\top x_j} \right \vert^2 ~ \mathrm{d}\omega,
\end{align}
which holds for any set of points $X = \{ x_1, ..., x_{|X|} \}$ and any translational invariant kernel $k(x, z) = \Phi(x-z)$.
Subsequently, $\hat{\Phi}(\omega)$ is bounded below by another function $\hat{\Psi}(\omega)$,
such that by using Eq.~\eqref{eq:kernel_vector_product} again in the reverse way, one obtains a diagonally dominant matrix,
for which the smallest eigenvalue can be bounded with standard tools like the Gershgorin circle theorem.

In \Cref{subsec:simple_proof} we present a simple proof of \Cref{th:stability_statement} based on a refined multivariate Ingham-type theorem that is presented in \Cref{subsec:refined_multivariate}.

\subsection{Norm equivalent kernels}
\label{subsec:norm_equiv_kernels}

In order to be able to derive results also for a broader class of kernels, that might not be translational invariant, we recall the notion of norm-equivalent kernels.
Two kernels $k_1$ and $k_2$, both defined on the same set $\Omega$ are said to be norm-equivalent,
if their RKHS $\mathcal{H}_{k_1}(\Omega)$ and $\mathcal{H}_{k_2}(\Omega)$ coincide (as sets) and the corresponding RKHS norms are norm-equivalent.
For these kernels, it is possible to derive a simple characterization of the Rayleigh quotients of kernel matrices, see e.g.~\cite[Proposition F.8]{haas2023mind}.

\begin{theorem}
\label{th:equivalent_kernels}
Let $k_1, k_2: \Omega \times \Omega \rightarrow \R$ be kernels such that their RKHS $\mathcal{H}_{k_1}(\Omega)$ and $\mathcal{H}_{k_2}(\Omega)$ coincide as sets and their norms satisfy
\begin{align*}
C^{-1} \cdot \Vert \cdot \Vert_{\mathcal{H}_{k_1}(\Omega)} \leq \Vert \cdot \Vert_{\mathcal{H}_{k_2}(\Omega)} \leq C \Vert \cdot \Vert_{\mathcal{H}_{k_1}(\Omega)}
\end{align*}
for some $C > 0$.
Then for any pairwise distinct points $X \subset \Omega$,
the kernel matrices $A_1 = k_1(X, X)$ and $A_2 = k_2(X, X)$ satisfy
\begin{align*}
C^{-2} \cdot \frac{\alpha^\top A_1 \alpha}{\Vert \alpha \Vert^2} ~\leq~ \frac{\alpha^\top A_2 \alpha}{\Vert \alpha \Vert^2} ~\leq~ C^2 \cdot \frac{\alpha^\top A_1 \alpha}{\Vert \alpha \Vert^2} \qquad \forall 0 \neq \alpha \in \R^{|X|}
\end{align*}
\end{theorem}

The importance of \Cref{th:equivalent_kernels} is,
that it is possible to change the kernel to an equivalent kernel,
and the spectral properties of the kernel matrices still behave in an equivalent way.
This becomes especially apparent, 
if one recalls the min-max eigenvalue characterization (Courant-Fischer characterization)
\begin{align}
\label{eq:courant_fisher}
\lambda_i(A_1) = \min_{V_i \subset \R^d} \max_{0 \neq \alpha \in V_i} \frac{\alpha^\top A_1 \alpha}{\Vert \alpha \Vert^2},
\end{align}
where the minimum is taken over all $i$-dimensional subspaces $V_i \subset \R^d$.

One of the results of the present paper, 
namely \Cref{th:main_th},
extends \Cref{th:equivalent_kernels} from norm-equivalent kernels to kernels of different smoothnesses.
This smoothness will be characterized in terms of the decay rate $\tau$ of the Fourier transform as in Eq.~\eqref{eq:fourier_decay}.

Due to \Cref{th:equivalent_kernels}, for simplicity we frequently focus on the case 
\begin{align*}
\hat{\Phi}(\omega) = (1+\Vert \omega \Vert^2)^{-\tau},
\end{align*}
i.e.\ equality in Eq.~\eqref{eq:fourier_decay}.
It is clear that the results in the following generalize in a straightforward way to the slightly more general case given in Eq.~\eqref{eq:fourier_decay},
and we include the corresponding statements as corollaries.

\begin{table}
\caption{Overview of translational invariant Matérn kernels $k(x, z) = \Phi(x-z)$ used for the numerical experiments.}
\label{tab:overview_matern}
\begin{center}
\begin{tabular}{ |c|c|c|c } 
 \hline
			 			& $\Phi(x)$ 				& $\hat{\Phi}(\omega)$ \\ \hline
 Basic Matérn 			& $\exp(-\Vert x \Vert)$ 			& $\propto (1+\Vert \omega \Vert^2)^{-\frac{d+1}{2}}$ \\ 
 Linear Matérn 			& $(1+\Vert x \Vert)\exp(-\Vert x \Vert)$ 		& $\propto (1+\Vert \omega \Vert^2)^{-\frac{d+3}{2}}$ \\ 
 Quadratic Matérn		& $(3+3\Vert x \Vert+\Vert x \Vert^2)\exp(-\Vert x \Vert)$ 	& $\propto (1+\Vert \omega \Vert^2)^{-\frac{d+5}{2}}$ \\ \hline
\end{tabular}
\end{center}
\end{table}

\section{Estimates on symbol functions}
\label{sec:symbol_funcs}

\subsection{Refined multivariate Ingham-type theorem}
\label{subsec:refined_multivariate}

We start by stating and proving a multivariate Ingham-type inequality \cite[Chapter 8]{komornik2005fourier}.
Compared to the results in the literature, we track the dependence on the size $R$ of the considered ball $B_R(x_0)$ more explicitly:

\begin{theorem}
\label{th:ingham_upper_refined}
Let $X \subset \R^d$ be a finite set of pairwise distinct points. 
Let $\lambda_{\min}(-\Delta)$ be the smallest eigenvalue of the Dirichlet Laplace operator on the $d$-dimensional ball $B_{1/2}(0) \subset \R^d$.

Then for any $R \geq \frac{c_0}{q_{X}} := \sqrt{2}\lambda_{\min}^{1/2}(-\Delta) q_X^{-1}$
and any $x_0 \in \R^d, \alpha = (\alpha_j)_{j=1}^{|X|} \in \R^{|X|}$ it holds 
\begin{align*}
c_1 R^d \cdot \Vert \alpha \Vert^2
\leq 
\int_{B_R(x_0)} \left| \sum_{j=1}^{|X|} \alpha_j e^{i\omega^\top x_j} \right|^2 ~ \mathrm{d}\omega 
\leq 
c_2 \cdot R^d \cdot \Vert \alpha \Vert^2
\end{align*}
with constants $c_1, c_2 > 0$ specified in Eq.~\eqref{eq:constants_c1_c2} depending only on $d \in \N$.
The second inequality already holds for $R \geq \frac{\pi}{q_X}$.

\end{theorem}

The proof is provided in the appendix, see \Cref{sec:proof_ingham}. 
The eigenvalues of the Dirichlet Laplace operator on a $d$-dimensional ball can be characterized with help of zeroes of Bessel functions \cite[Theorem 3.61]{frank2022schrodinger},
which are well researched \cite[Chapter 15]{watson1922treatise}.
Thus we obtain for the smallest eigenvalue $\lambda_{\min}(-\Delta)$ of the Dirichlet Laplace operator on $B_{1/2}(0) \subset \R^d$ 
the values $\lambda_{\min}(-\Delta) = \pi^2 \approx 9.869$ for $d=1$, $\lambda_{\min}(-\Delta) \approx 23.132$ for $d=2$, $\lambda_{\min}(-\Delta) = 4 \cdot \pi^2 \approx 39.478$ for $d=3$ and $\lambda_{\min}(-\Delta) \approx 58.727$ for $d=4$.

For general $d \geq 3$ we have the asymptotic
\begin{align}
\label{eq:estimate_lambda_min2}
d^2-4 < \lambda_{\min}(-\Delta) < 2d(d+4).
\end{align}
Also the constants $c_1$ and $c_2$ can be evaluated explicitly,
and one obtains for $d=1$ the values $c_1 = \frac{1}{2\sqrt{\pi}} \approx 0.159$ and $c_2 = \frac{4}{\pi} \approx 1.273$.

\subsection{A simple proof for a lower bound on $\lambda_{\min}(A)$}
\label{subsec:simple_proof}

In the following we provide a simple proof for lower bounding the smallest eigenvalue of kernel matrices.
In fact, \Cref{th:lower_bound_easy} derives the same bounds as known in the literature (see e.g.~\cite[Theorem 12.3]{wendland2005scattered}),
however with a simple and straightforward proof based on \Cref{th:ingham_upper_refined}.

\begin{theorem}
\label{th:lower_bound_easy}
Let $k(x, z) = \Phi(x - z)$ be a translational invariant kernel with $\Phi \in L^1(\R^d)$.
Given pairwise distinct points $X \subset \R^d$,
the smallest eigenvalue of the kernel matrix $A = k(X, X) \in \R^{|X| \times |X|}$
can be bounded below as
\begin{align}
\label{eq:lower_bound_easy}
\lambda_{\min}(A) \geq \frac{c_0^d c_1}{(2\pi)^{d/2}} \cdot \min_{\omega \in B_{c_0 / q_X}(0)} \hat{\Phi}(\omega)  q_X^{-d}
\end{align}
with constants $c_0, c_1 > 0$ only depending on $d$.
\end{theorem}

\begin{proof}
We make use of \Cref{th:ingham_upper_refined} and pick $R:= \frac{c_0}{q_X}$.
We express the vector matrix vector product $\alpha^\top A \alpha$ with help of the Fourier transform 
according to Eq.~\eqref{eq:kernel_vector_product} and estimate in a straightforward way
\begin{align*}
\alpha^\top A \alpha &= (2\pi)^{-d/2} \cdot \int_{\R^d} \hat{\Phi}(\omega) \cdot \left \vert \sum_{j=1}^{|X|} \alpha_j e^{i \omega^\top x_j} \right \vert^2 ~ \mathrm{d}\omega \\
&\geq (2\pi)^{-d/2} \cdot \int_{B_R} \hat{\Phi}(\omega) \cdot \left \vert \sum_{j=1}^{|X|} \alpha_j e^{i \omega^\top x_j} \right \vert^2 ~ \mathrm{d}\omega \\
&\geq (2\pi)^{-d/2} \cdot \min_{\omega \in B_R(0)} \hat{\Phi}(\omega) \cdot \int_{B_R}  \left \vert \sum_{j=1}^{|X|} \alpha_j e^{i \omega^\top x_j} \right \vert^2 ~ \mathrm{d}\omega \\
&\geq (2\pi)^{-d/2} \cdot c_1 \min_{\omega \in B_R(0)} \hat{\Phi}(\omega)  R^d \cdot \Vert \alpha \Vert^2.
\end{align*}
Thus, the Rayleigh characterization of eigenvalues immediately yields
\begin{align*}
\lambda_{\min}(A) = \min_{0 \neq \alpha \in \R^{|X|}} \frac{\alpha^\top A \alpha}{\Vert \alpha \Vert^2}
\geq \frac{c_1}{(2\pi)^{d/2}} \cdot \min_{\omega \in B_R(0)} \hat{\Phi}(\omega)  R^d.
\end{align*}
Recalling the choice of $R$ as $R \equiv \frac{c_0}{q_X}$ gives the result.
\end{proof}

The simplicity of the proof of \Cref{th:lower_bound_easy} relies on the explicit availability of the dependency in $R$ within \Cref{th:ingham_upper_refined},
and thus highlights the importance of the refinement provided in \Cref{th:ingham_upper_refined} over the literature.

By inserting the Fourier transform for specific translational invariant kernels $k$, 
we obtain precise bounds:

\begin{cor}
\label{cor:minimal_eigenvalue_bound}
Given pairwise distinct points $X \subset \R^d$,
the smallest eigenvalue of the kernel matrix $A = k(X, X) \in \R^{|X| \times |X|}$ for any kernel satisfying Eq.~\eqref{eq:fourier_decay} can be bounded below  as
\begin{align}
\label{eq:lambda_min_sobolev2}
\lambda_{\min}(A) ~\geq~ \frac{c_0^d c_1 c_\Phi}{(2\pi)^{d/2}} \cdot  (q_X^2+c_0^2)^{-\tau} \cdot q_X^{2\tau-d}.
\end{align}
\end{cor}

We conclude this subsection with two remarks.
First, if $X \subset \Omega \subset \R^d$ for $\Omega$ bounded, 
we have $q_X$ bounded and thus Eq.~\eqref{eq:lambda_min_sobolev2} simplifies to Eq.~\eqref{eq:lambda_min_sobolev1},
i.e.\ we obtain the same result as in \Cref{th:stability_statement} with a more simple proof.
As the case $\Omega \subset \R^d$ bounded appears often, 
we will assume $\Omega$ to be bounded in the following theorems,
such that we can make use of the simplified bound as in Eq.~\eqref{eq:lambda_min_sobolev1}.
We note that for the case of uniformly distributed points $X \subset \Omega$, 
there is a corresponding matching upper bound, see e.g.\ \cite[Proposition 4.1]{santin2024optimality}.

We continue with a remark on the Gaussian kernel.

\begin{rem}
In the setting of \Cref{cor:minimal_eigenvalue_bound},
but inserting the Fourier transform for the Gaussian kernel $k(x, z) = \exp(-\gamma \Vert x - z \Vert^2)$, $\gamma > 0$ gives
\begin{align}
\label{eq:lambda_min_gaussian2}
\lambda_{\min}(A) ~\geq~ \frac{c_0^d c_1}{(2\pi)^{d/2}} \cdot \left( \frac{\pi}{\gamma} \right)^{d/2} \cdot q_X^{-d} \exp \left(- \frac{c_0^2}{4\gamma} \cdot q_X^{-2} \right)
\end{align}
Using $c_0 \equiv \sqrt{2}\lambda_{\min}^{1/2}(-\Delta)$ from \Cref{th:ingham_upper_refined} and using the estimates on $\lambda_{\min}^{1/2}(-\Delta)$ from Eq.~\eqref{eq:estimate_lambda_min2} (for $d \geq 3$) we have 
\begin{align*}
2(d^2-4) < c_0^2 < 4d(d+4)
\end{align*}
Like this we can further estimate Eq.~\eqref{eq:lambda_min_gaussian2} as
\begin{align*}
\lambda_{\min}(A) \geq \frac{c_0^d c_1}{(2\pi)^{d/2}} \cdot \left( \frac{\pi}{\gamma} \right)^{d/2} \cdot q_X^{-d} \exp \left(-\frac{d(d+4)}{\gamma} \cdot q_X^{-2} \right),
\end{align*}
which is better than the bound given in \cite[Corollary 12.4]{wendland2005scattered} and close to the specialized bound in \cite[Eq.~(14)]{diederichs2019improved}.
Following \cite[Remark~2.5]{diederichs2019improved} and tweaking some constants (e.g.\ $c_0 \equiv \sqrt{2} \lambda_{\min}^{1/2}(-\Delta)$ can be reduced to $(1+\varepsilon) \lambda_{\min}^{1/2}(-\Delta)$ for any $\varepsilon > 0$ 
by the choice $\tilde{r}^2 = (1+\varepsilon)R^2 \lambda_{\min}(-\Delta)$ in the proof of \Cref{th:ingham_upper_refined}) we can recover the rate presented in \cite[Eq.~(14)]{diederichs2019improved}.
\end{rem}

All in all, the fact that we recover the sharp results from the literature in \Cref{th:lower_bound_easy} and \Cref{cor:minimal_eigenvalue_bound} shows the usefulness of the approach via the refined multivariate Ingham-type theorem as given in \Cref{th:ingham_upper_refined}.

\section{Novel and refined stability estimates}
\label{sec:novel_refined_estimates}

\subsection{Localization estimate}

The following \Cref{th:localization_integral} shows,
that the largest contribution of the integrals from Eq.~\eqref{eq:kernel_vector_product} is localized in a (large) ball $B_{2R} := B_{2R}(0)$ around the origin.
We exemplify this for the case of Sobolev kernels, i.e.\ the case where $\hat{\Phi}(\omega)$ satisfies Eq.~\eqref{eq:fourier_decay}.
The theorem is formulated for $\hat{\Phi}(\omega) = (1+\Vert \omega \Vert^2)^{-\tau}$ for some $\tau > d/2$,
while the extension to the equivalence as in Eq.~\eqref{eq:fourier_decay} is trivial.

\begin{theorem}
\label{th:localization_integral}
Let $\Omega \subset \R^d$ be bounded.
For $\tau > d/2$ and any $\varepsilon \in (0, 1)$, 
there exists a constant $a := a_\varepsilon > 0$ only depending on $\varepsilon, d, \tau, \Omega$,
such that for any $X \subset \Omega \subset \R^d$ of finitely many pairwise distinct points 
and any $\alpha = (\alpha_j)_{j=1}^{|X|} \in \R^d$
we have for any $R \geq \frac{\max(a, \pi)}{q_{X}}$:
\begin{align*}
(1-\varepsilon) \int_{\R^d} (1+\Vert \omega \Vert^2)^{-\tau} \left \vert \sum_{j=1}^{|X|} \alpha_j e^{i \omega^\top x_j} \right \vert^2 \mathrm{d}\omega
\leq
\int_{B_{2R}} (1+\Vert \omega \Vert^2)^{-\tau} \left \vert \sum_{j=1}^{|X|} \alpha_j e^{i \omega^\top x_j} \right \vert^2 \mathrm{d}\omega.
\end{align*}
\end{theorem}

\begin{proof}
The statement for $\varepsilon = \frac{1}{4}$ can be found in \cite[Lemma 3.3]{narcowich2006sobolev},
and the generalization to $\varepsilon \in (0, 1)$ is straightforward.
\end{proof}

We remark that it is possible to derive results like \Cref{th:localization_integral} also for kernels of infinite smoothness such as the Gaussian kernels.
As we focus on finitely smooth kernels in this work, we do not include the details here.

\subsection{Spectral alignment of kernel matrices} 
\label{subsec:spectral_alignment}

In this subsection, we focus on kernels of finite smoothness.
In particular we start by considering a translational invariant
\begin{align*}
\text{kernel $k$ that satisfies \quad} \hat{\Phi}(\omega) = (1+\Vert \omega \Vert^2)^{-\tau} ~ \forall \omega \in \R^d
\end{align*}
as well as another translational invariant 
\begin{align*}
\text{kernel $k^{(\sigma)}$ that satisfies \quad} \hat{\Phi}^{(\sigma)}(\omega) = (1+\Vert \omega \Vert^2)^{-\sigma \tau} ~ \forall \omega \in \R^d
\end{align*}
for some $\sigma \in (\frac{d}{2\tau}, 1)$.
Possibly up to some norm equivalences as in Eq.~\eqref{eq:fourier_decay} and discussed in \Cref{subsec:norm_equiv_kernels}, 
this corresponds to the use of two kernels of finite but different smoothness.
We consider the following example:

\begin{example}
\label{ex:matern_kernels}
Let $k$ be the linear Matérn kernel, which is translational invariant with $\Phi_{\text{lin}}(x) = (1+\Vert x \Vert)\exp(-\Vert x \Vert)$ which satisfies $\hat{\Phi}(\omega) \doteq (1+\Vert \omega \Vert^2)^{-\tau_{\text{lin}}}$ ($\doteq$ means equality up to a constant) with $\tau_{\text{lin}} = \frac{d+3}{2}$.
We choose $\sigma = \frac{d+1}{d+3}$, i.e.\ $\sigma \tau_\text{lin} = \frac{d+1}{2}$.
This smoothness corresponds to the basic Matérn kernel, which is translational invariant with $\Phi_{\text{basic}}(x) = \exp(-\Vert x \Vert)$ and $\hat{\Phi}_\text{basic}(\omega) \doteq (1 + \Vert \omega \Vert^2)^{-\tau_\text{basic}}$ with $\tau_\text{basic} = \frac{d+1}{2} = \sigma \tau_\text{lin}$. 
Thus for this choice of $\sigma$, $k^{(\sigma)}$ is the basic Matérn kernel.
\end{example}

Enabled by \Cref{th:ingham_upper_refined} and \Cref{th:localization_integral}, we can now derive the following main result \Cref{th:main_th}, which relates the Rayleigh quotients for kernel matrices for $k$ and $k^{(\sigma)}$.
It can be understood as a statement on the spectral behaviour of the kernel matrices $A := k(X, X)$ and $A^{(\sigma)} := k^{(\sigma)}(X, X)$ and as a generalization of \Cref{th:equivalent_kernels} from norm-equivalent kernels to kernels of different smoothness:

\begin{theorem}[Main result]
\label{th:main_th}
Let $\Omega \subset \R^d$ and consider any finitely many pairwise distinct points $X \subset \Omega$ and let $\tau > d/2$, $\sigma \in (\frac{d}{2\tau}, 1]$.
Consider the kernel matrices $A = k(X, X)$ and $A^{(\sigma)} = k^{(\sigma)}(X, X)$ of the kernels $k$ and $k^{(\sigma)}$ as specified before.
Then it holds for all $0 \neq \alpha \in \R^{|X|}$

\begin{align}
\label{eq:main_result}
\frac{\alpha^\top A \alpha}{\Vert \alpha \Vert^2} ~ &\leq ~
\frac{\alpha^\top A^{(\sigma)} \alpha}{\Vert \alpha \Vert^2} 
\leq ~ c_\sigma \left( q_X^{2\tau} + c_0^2 \right)^{(1-\sigma)\tau} q_X^{-(1-\sigma)2\tau} \cdot \frac{\alpha^\top A \alpha}{\Vert \alpha \Vert^2}
\end{align}
with a constant $c_\sigma > 0$ (only depending on $\sigma$, $d$, $\tau$, $\Omega$) as defined in Eq.~\eqref{eq:constant_sigma}.
For $\Omega \subset \R^d$ bounded, this simplifies to
\begin{align}
\label{eq:main_result2}
\frac{\alpha^\top A \alpha}{\Vert \alpha \Vert^2} ~ &\leq ~
\frac{\alpha^\top A^{(\sigma)} \alpha}{\Vert \alpha \Vert^2} 
\leq ~ c_\sigma' q_X^{-(1-\sigma)2\tau} \cdot \frac{\alpha^\top A \alpha}{\Vert \alpha \Vert^2} \qquad \forall 0 \neq \alpha \in \R^{|X|}.
\end{align}
\end{theorem}

\begin{proof}
For $\sigma = 1$ the statement is trivial as $k = k^{(\sigma)}$ and thus $A = A^{(\sigma)}$.
For $\sigma \in (\frac{d}{2\tau}, 1)$ we make use of the representation of the expressions $\alpha^\top A \alpha$ and $\alpha^\top A^{(\sigma)} \alpha$ in Fourier space according to Eq.~\eqref{eq:kernel_vector_product}:

The lower bound is straightforward:
\begin{align*}
\alpha^\top A^{(\sigma)} \alpha &= (2\pi)^{-d/2} \cdot \int_{\R^d} (1+\Vert \omega \Vert^2)^{-\sigma \tau} \cdot \left \vert \sum_{j=1}^{|X|} \alpha_j e^{i \omega^\top x_j} \right \vert^2 ~ \mathrm{d}\omega \\
&\geq (2\pi)^{-d/2} \cdot \int_{\R^d} (1+\Vert \omega \Vert^2)^{-\tau} \cdot \left \vert \sum_{j=1}^{|X|} \alpha_j e^{i \omega^\top x_j} \right \vert^2 ~ \mathrm{d}\omega \\
&= \alpha^\top A \alpha.
\end{align*}
The upper bound is more sophisticated:
We use \Cref{th:localization_integral} for $R = \frac{\max(a, \pi)}{q_{X}}$ and $\varepsilon = \frac{1}{2}$ and subsequently apply Hölder inequality with exponents $p = \frac{1}{\sigma} \in (1, \infty)$, $q = \frac{1}{1-\sigma} \in (1, \infty)$:
\begin{align*}
& ~ (2\pi)^{d/2} \cdot \alpha^\top A^{(\sigma)} \alpha \\
=& ~ \int_{\R^d} \left( 1+\Vert \omega \Vert^2 \right)^{-\sigma\tau} \cdot \left| \sum_{j=1}^{|X|} \alpha_j e^{i\omega^\top x_j} \right|^2 ~ \mathrm{d}\omega \\
\leq& ~ 2 \cdot \int_{B_{2R}} \left( 1+\Vert \omega \Vert^2 \right)^{-\sigma\tau} \cdot \left| \sum_{j=1}^{|X|} \alpha_j e^{i\omega^\top x_j} \right|^2 ~ \mathrm{d}\omega \\
=& ~ 2 \cdot \int_{B_{2R}} \left( 1+\Vert \omega \Vert^2 \right)^{-\sigma\tau} \cdot \left| \sum_{j=1}^{|X|} \alpha_j e^{i\omega^\top x_j} \right|^{2\sigma} \cdot \left| \sum_{j=1}^{|X|} \alpha_j e^{i\omega^\top x_j} \right|^{2-2\sigma} ~ \mathrm{d}\omega \\
\leq& ~ 2 \cdot \left( \int_{B_{2R}} \left( 1+\Vert \omega \Vert^2 \right)^{-\tau} \cdot \left| \sum_{j=1}^{|X|} \alpha_j e^{i\omega^\top x_j} \right|^{2} ~ \mathrm{d}\omega \right)^{\sigma} \cdot \left( \int_{B_{2R}} \left| \sum_{j=1}^{|X|} \alpha_j e^{i\omega^\top x_j} \right|^2 ~ \mathrm{d}\omega \right)^{1-\sigma} \\
\leq& ~ 2 \cdot \left( \int_{\R^d} \left( 1+\Vert \omega \Vert^2 \right)^{-\tau} \cdot \left| \sum_{j=1}^{|X|} \alpha_j e^{i\omega^\top x_j} \right|^{2} ~ \mathrm{d}\omega \right)^{\sigma} \cdot \left( c_2 2^dR^d \Vert \alpha \Vert^2 \right)^{1-\sigma} \\
=& ~ 2c_2^{1-\sigma} \cdot \left( \int_{\R^d} \left( 1+\Vert \omega \Vert^2 \right)^{-\tau} \cdot \left| \sum_{j=1}^{|X|} \alpha_j e^{i\omega^\top x_j} \right|^{2} ~ \mathrm{d}\omega \right)^{\sigma} \cdot \Vert \alpha \Vert^{2(1-\sigma)} \cdot R^{d(1-\sigma)} \\
=& ~ 2^{1+d(1-\sigma)}c_2^{1-\sigma} (2\pi)^{\sigma \cdot d/2} \left( \alpha^\top A \alpha \right)^{\sigma} \cdot \Vert \alpha \Vert^{2(1-\sigma)} \cdot R^{d(1-\sigma)}.
\end{align*}
where \Cref{th:ingham_upper_refined} was additionaly used for the last inequality.
By recalling that $R = \frac{\max(a, \pi)}{q_{X}}$,
we obtain
\begin{align*}
(2\pi)^{d/2} \cdot \alpha^\top A^{(\sigma)} \alpha \leq& ~ 2^{1+d(1-\sigma)}c_2^{1-\sigma} (2\pi)^{\sigma \cdot d/2} \max(a, \pi)^{d(1-\sigma)} q_{X}^{-d(1-\sigma)} (\alpha^\top A \alpha)^{\sigma} \cdot \Vert \alpha \Vert^{2(1-\sigma)} \\ 
\Leftrightarrow \quad \frac{\alpha^\top A^{(\sigma)} \alpha}{\Vert \alpha \Vert^2} \leq& ~ 2^{1+d(1-\sigma)}c_2^{1-\sigma} (2\pi)^{-(1-\sigma)\cdot d/2} \max(a, \pi)^{d(1-\sigma)} \cdot q_{X}^{-d(1-\sigma)} \left( \frac{\alpha^\top A \alpha}{\Vert \alpha \Vert^2} \right)^{\sigma}.
\end{align*}
Finally we use Eq.~\eqref{eq:lambda_min_sobolev2} to further estimate 
\begin{align*}
\left( \frac{\alpha^\top A \alpha}{\Vert \alpha \Vert^2} \right)^{\sigma} &= \left( \frac{\alpha^\top A \alpha}{\Vert \alpha \Vert^2} \right) \cdot \left( \frac{\alpha^\top A \alpha}{\Vert \alpha \Vert^2} \right)^{\sigma-1} \\
&\leq \left( \frac{\alpha^\top A \alpha}{\Vert \alpha \Vert^2} \right) \cdot \left( \frac{c_0^d c_1}{(2\pi)^{d/2}} \cdot \left( q_X^{2\tau} + c_0^2 \right)^{-\tau} q_X^{2\tau - d} \right)^{\sigma-1}.
\end{align*}
Thus altogether this results in
\begin{align*}
\frac{\alpha^\top A^{(\sigma)} \alpha}{\Vert \alpha \Vert^2} \leq c_\sigma \left( q_X^{2\tau} + c_0^2 \right)^{\tau(1-\sigma)} q_X^{-(1-\sigma)2\tau} \cdot \frac{\alpha^\top A \alpha}{\Vert \alpha \Vert^2}
\end{align*}
using the constant 
\begin{align}
\label{eq:constant_sigma}
c_\sigma := 2^{1+d(1-\sigma)} c_2^{1-\sigma} (2\pi)^{-(1-\sigma)\cdot d/2} \max(a, \pi)^{d(1-\sigma)} \left(\frac{c_0^d c_1}{(2\pi)^{d/2}} \right)^{-(1-\sigma)}.
\end{align}
If $\Omega \subset \R^d$ is additionally bounded, 
then $q_X$ is bounded as well,
and one obtains the simplification by the choice $c_\sigma' := c_\sigma \cdot \sup_{X \subset \Omega} \left( q_X^{2\tau} + c_0^2 \right)^{(1-\sigma)\tau}$.
\end{proof}

On the one hand, \Cref{th:main_th} can be seen as a generalization of \Cref{th:equivalent_kernels} from norm-equivalent kernels to kernels of different smoothness.
In fact, when using the limit value $\sigma = 1$ in \Cref{th:main_th}, 
we obtain the same result as \Cref{th:equivalent_kernels} (up to constants due to norm-equivalences).
On the other hand, \Cref{th:main_th} can also be understood as a deterministic counterpart to \cite[Theorem 3.2]{geifman2024controlling},
which considers so-called modified spectrum kernels in a probabalistic setting.
We remark that a generalization of \Cref{th:main_th} to conditionally positive definite kernels such as polyharmonic splines \cite{duchon1977splines} is feasible,
as they can also be treated with help of the (generalized) Fourier transform.
We did not include this significantly more technical discussion here to keep the outline more simple,
and point to \cite{hangelbroek2025extending} for related work on the spectra of conditionally positive definite kernels.

Within \Cref{th:main_th},
we also included the case of $\Omega$ being unbounded.
In this case, the point $X$ can be far apart, such that $q_X \rightarrow \infty$,
which therefore implies $\left( q_X^{2\tau} + c_0^2 \right)^{(1-\sigma)\tau} q_X^{-(1-\sigma)2\tau} \rightarrow 1$.
This makes sense, because the matrices $A$ and $A^{(\sigma)}$ become diagonal dominant.
In the following, we will thus focus on the more interesting case of $\Omega$ being bounded.

Finally we comment on a spectral alignment phenomenon.
We consider points $X$ such that their separation distance $q_X$ is not too small,
which is frequently the case in kernel based approximation when considering well distributed points.
Then let $\alpha$ be an eigenvector to a small eigenvalue of the matrix $A$.
Then the upper bound in Eq.~\eqref{eq:main_result} shows that this vector $\alpha$ provides also a small value for the Rayleigh quotient $\frac{\alpha^\top A^{(\sigma)} \alpha}{\Vert \alpha \Vert^2}$ of the matrix $A^{(\sigma)}$.
On the other side, an eigenvector to a large eigenvalue of $A^{(\sigma)}$ can never result in a small value of the Rayleigh quotient $\frac{\alpha^\top A \alpha}{\Vert \alpha \Vert^2}$ of the matrix $A$.
Also the lower bound of Eq.~\eqref{eq:main_result} allows for similar implications:
If $\alpha$ is an eigenvector to a large eigenvalue of the matrix $A$,
then the lower bound in Eq.~\eqref{eq:main_result} shows that the Rayleigh quotient of the matrix $A^{(\sigma)}$ will also be large.
However, note that these implications do not necessarily imply that eigenvectors for the matrices $A$ and $A^{(\sigma)}$ need to be similar.
This becomes obvious, if one considers a case of duplicate eigenvalues, where there is an eigenspace of dimension larger than one.
Numerical experiments on this behaviour are presented in \Cref{sec:numerical_experiments}. %

We conclude with two corollaries.
First, \Cref{cor:extension_to_norm_equiv_kernels} extends \Cref{th:main_th} to norm-equivalent kernels,
i.e.\ we only assume to be norm-equivalent to kernels that satisfy $\hat{\Phi}(\omega) = (1+\Vert \omega \Vert^2)^{-\tau}$.
This allows us to apply the result also to e.g.\ Wendland kernels, 
as examplified subsequently in \Cref{ex:wendland_kernel}.

\begin{cor}
\label{cor:extension_to_norm_equiv_kernels}
Let the assumptions of \Cref{th:main_th} hold, 
and suppose in addition that $\Omega$ is bounded.
Let $\tilde{k}, \tilde{k}^{(\sigma)}$ be norm-equivalent to $k$ respective $k^{(\sigma)}$, i.e.\
\begin{align*}
C^{-1} \cdot \Vert \cdot \Vert_{\mathcal{H}_k(\Omega)} &\leq \Vert \cdot \Vert_{\mathcal{H}_{\tilde{k}}(\Omega)} \leq C \Vert \cdot \Vert_{\mathcal{H}_k(\Omega)}, \\
C_1^{-1} \cdot \Vert \cdot \Vert_{\mathcal{H}_{k^{(\sigma)}}(\Omega)} &\leq \Vert \cdot \Vert_{\mathcal{H}_{\tilde{k}^{(\sigma)}}(\Omega)} \leq C_1 \Vert \cdot \Vert_{\mathcal{H}_{k^{(\sigma)}}(\Omega)},
\end{align*}
Consider the kernel matrices $\tilde{A} = \tilde{k}(X, X)$ and $\tilde{A}^{(\sigma)} = \tilde{k}^{(\sigma)}(X, X)$.
Then it holds 
\begin{align*}
C_1^{-2} C^{-2} \cdot \frac{\alpha^\top \tilde{A} \alpha}{\Vert \alpha \Vert^2} ~ &\leq ~
\frac{\alpha^\top \tilde{A}^{(\sigma)} \alpha}{\Vert \alpha \Vert^2} 
\leq ~ C_1^2 C^2 \cdot c_\sigma' q_X^{-(1-\sigma)2\tau} \cdot \frac{\alpha^\top \tilde{A} \alpha}{\Vert \alpha \Vert^2} \qquad \forall 0 \neq \alpha \in \R^{|X|}.
\end{align*}
\end{cor}

\begin{proof}
This follow immediately from \Cref{th:main_th} by taking into account the norm-equivalences and \Cref{th:equivalent_kernels}.
\end{proof}

Note that the kernel $k$ and $k^{(\sigma)}$ within \Cref{cor:extension_to_norm_equiv_kernels} do not necessarily need to be translational invariant anymore,
only the norm-equivalence is assumed.
We conclude with the following example. 
It extends \Cref{ex:matern_kernels} to norm-equivalent kernels such as the family of Wendland kernels.

\begin{example}
\label{ex:wendland_kernel}
As in \Cref{ex:matern_kernels}, let $k$ be the linear Matérn kernel, 
i.e.\ $\hat{\Phi}(\omega) \stackrel{\cdot}{=} (1 + \Vert \omega \Vert^2)^{-\tau_{\text{lin}}}$ with $\tau_{\text{lin}} = \frac{d+3}{2}$.
Recall that the Wendland kernels $\Phi_{d, k}$ of smoothness $k \in \N$ in dimension $d \in \N$ exhibits an asymptotic Fourier decay as $\hat{\Phi}(\omega) \asymp (1+\Vert \omega \Vert^2)^{-\frac{d}{2}-k-\frac{1}{2}}$ ($d\geq 3$ if $k=0$) \cite[Theorem 10.35]{wendland2005scattered}.
In view of \Cref{cor:extension_to_norm_equiv_kernels} and using $d=3$,
the choice of $\sigma = 2/3$ yields $\sigma \cdot \frac{d+3}{2} = 2 \stackrel{!}{=} -\frac{d}{2} - k - \frac{1}{2}$, i.e.\ $k=0$.
Thus for this choice of $d$ and $\sigma$, $\tilde{k}^{(\sigma)}$ is the Wendland kernel $\Phi_{3,0}$.
\end{example}

As a second corollary we provide a worst case bound on the ratio $\frac{\alpha^\top \tilde{A}^{(\sigma)} \alpha}{\alpha^\top \tilde{A} \alpha}$.
This can be understood as a stability result,
as it immediately yields an upper bound on the condition number $\mathrm{cond}(\tilde{A}^{-1/2} \tilde{A}^{(\sigma)} \tilde{A}^{-1/2})$, 
see \Cref{rem:condition_number}.
\begin{cor}
\label{cor:worst_case_bound}
Under the assumptions of \Cref{cor:extension_to_norm_equiv_kernels}
it holds the worst case bound
\begin{align*}
C_1^{-2} C^{-2} \leq \min_{0 \neq \alpha \in \R^{|X|}} \frac{\alpha^\top \tilde{A}^{(\sigma)} \alpha}{\alpha^\top \tilde{A} \alpha} \leq \max_{0 \neq \alpha \in \R^{|X|}} \frac{\alpha^\top \tilde{A}^{(\sigma)} \alpha}{\alpha^\top \tilde{A} \alpha} \leq&
~ C_1^2 C^2 c_\sigma' q_{X}^{-(1-\sigma)2\tau}.
\end{align*}
Furthermore the ordered eigenvalues $\lambda_1(\tilde{A}) \leq ... \leq \lambda_{|X|}(\tilde{A})$ of $\tilde{A}$ and the eigenvalues $\lambda_1(\tilde{A}^{(\sigma)}) \leq ... \leq \lambda_{|X|}(\tilde{A}^{(\sigma)})$ of $\tilde{A}^{(\sigma)}$ satisfy the bounds
\begin{align*}
C_1^{-2} C^{-2} \lambda_i(\tilde{A}) ~\leq~ \lambda_i(\tilde{A}^{(\sigma)}) ~\leq~ c_\sigma' q_{X}^{-(1-\sigma)2\tau} \cdot \lambda_i(\tilde{A}).
\end{align*}
with $C \equiv c_2 a^d c_0^{-1}$ and thus only depending on $d, \tau$ and $\Omega$.
\end{cor}

\begin{proof}
The first part follows from rearranging the result of \Cref{cor:extension_to_norm_equiv_kernels},
while the second part follows in the same way by additionally applying the Courant-Fischer characterization of Eq.~\eqref{eq:courant_fisher}.
\end{proof}

\noindent It is worth to observe that the straightforward estimate 
\begin{align*}
\max_{\alpha \in \R^{|X|}} \frac{\alpha^\top \tilde{A}^{(\sigma)} \alpha}{\alpha^\top \tilde{A} \alpha} \leq \lambda_{\max}(\tilde{A}^{(\sigma)}) / \lambda_{\min}(\tilde{A})
\end{align*}
results in a scaling as $q_X^{-d} \cdot q_X^{d - 2\tau} = q_X^{-2\tau}$,
which is significantly worse than the estimate obtained in \Cref{cor:worst_case_bound}.

We conclude with the following remark:
\begin{rem}
\label{rem:condition_number}
The bounds of \Cref{cor:worst_case_bound} directly imply bounds on condition numbers as
\begin{align*}
\mathrm{cond}(\tilde{A}^{-1/2} \tilde{A}^{(\sigma)} \tilde{A}^{-1/2}) = \mathrm{cond}((\tilde{A}^{(\sigma)})^{-1/2} \tilde{A} (\tilde{A}^{(\sigma)})^{-1/2}) \leq C_1^4 C^4 c_\sigma' q_X^{-(1-\sigma)2\tau}.
\end{align*}
Thus, especially for $\sigma \approx 1$, 
the condition numbers of $\tilde{A}^{-1/2} \tilde{A}^{(\sigma)} \tilde{A}^{-1/2}$ as well as $(\tilde{A}^{(\sigma)})^{-1/2} \tilde{A} (\tilde{A}^{(\sigma)})^{-1/2}$ are significantly smaller than those of $\tilde{A}$, $\tilde{A}^{(\sigma)}$.
Such a property was the idea behind the numerical studies of \cite{egidi2023preconditioning} within a related setting for preconditioning purposes.
In a related setting, the use of such matrices was investigated numerically in \cite{egidi2023preconditioning} for preconditioning purposes.
\end{rem}

Finally we remark that it is possible to derive analogous results also for kernels of infinite smoothness such as the Gaussian kernels.
For the case of the Gaussian kernel with $\Phi(x) = \exp(-\gamma \Vert x \Vert^2)$ and shape parameter $\gamma > 0$,
the Fourier transform is given as $\hat{\Phi}(\omega) = (\pi/\gamma)^{d/2} \cdot \exp(-\Vert \omega \Vert^2 / (4\gamma))$.
Thus, different decay rates of the Fourier transform correspond to the use of different shape parameters of the Gaussian kernel.
Hence, such results seem interesting for investigating properties of the Gaussian kernel with respect to its shape parameter $\gamma$.
As we focus on finitely smooth kernels in this work, we leave out further details here.

\section{Numerical results on spectral alignment of kernel matrices}
\label{sec:numerical_experiments}

In this section we present numerical results revolving around the main result \Cref{th:main_th}.
For enhanced reproducibility of the numerical experiments,
we provide the code to rerun these experiments.\footnote{\url{https://gitlab.rrz.uni-hamburg.de/bbd9097/paper-2024-stability-estimates-kernel-matrices}}

As elaborated in \Cref{subsec:spectral_alignment}, 
\Cref{th:main_th} provides estimates on the Rayleigh quotients for kernel matrices for kernels of different smoothness.
As an example, the eigenvector to the smallest eigenvalue of the matrix $A$ also results in a small value when used in the Rayleigh quotient of the matrix $A^{(\sigma)}$.
On the other hand, 
the eigenvector to the largest eigenvalue of the matrix $A$ also results in a large value when used as vector in the Rayleigh quotient of the matrix $A^{(\sigma)}$.

In order to illustrate this behaviour, we consider the quadratic Matérn kernel and the basic Matérn kernel.
In view of Eq.~\eqref{eq:fourier_decay}, 
and depending on the dimension $d$,
the quadratic Matérn kernel has a smoothness of $\tau_{\text{quad}} = \frac{d+5}{2}$ and the basic Matérn kernel has a smoothness of $\tau_{\text{basic}} = \frac{d+1}{2}$.
Thus this choice of kernels correspond to a value of $\sigma$ as $\sigma = \frac{\tau_\text{basic}}{\tau_\text{quad}} = \frac{d+1}{d+5}$.
We consider two cases, namely $\Omega = [0, 1] \subset \R$ as well as $\Omega = [0, 1]^3 \subset \R^3$ and use 20 uniformly randomly sampled points in $\Omega$, 
such that $A, A^{(\sigma)} \in \R^{20 \times 20}$ in both cases.
For these two kernels, 
we consider the eigenvectors $\{ v_i \}_{i=1}^{20} \subset \R^{20}$ of the matrix $A$ (for the quadratic Matérn kernel) 
and the eigenvectors $\{ v_i^{(\sigma)} \}_{i=1}^{20} \subset \R^{20}$ of the matrix $A^{(\sigma)}$ (for the basic Matérn kernel).
The eigenvectors refer to the sorted eigenvalues, 
i.e.\ $\lambda_1 \leq ... \leq \lambda_{20}$ and $\lambda_1^{(\sigma)} \leq ... \leq \lambda_{20}^{(\sigma)}$.
We compare these eigenvectors in terms of the squared dot products 
\begin{align}
\label{eq:squared_dot_products}
|\langle v_i, v^{(\sigma)}_j \rangle|^2, ~~ i, j = 1, ..., 20.
\end{align}
Note that $\sum_{i=1}^{20} |\langle v_i, v^{(\sigma)}_j \rangle|^2 = \sum_{j=1}^{20} |\langle v_i, v^{(\sigma)}_j \rangle|^2 = 1$ due to Parseval's theorem.

The results are depicted as heatmap plots in \Cref{fig:spectral_alignment}:
The left plot refers to the $d=1$ case, while the right plots refers to the $d=3$ case.
One can clearly see that the values of the squared dot products from Eq.~\eqref{eq:squared_dot_products} are largest close to the diagonal.
This is explained by the two-sided inequality of Eq.~\eqref{eq:main_result}.

\begin{figure}[t]
\begin{center}
\setlength\fwidth{.47\textwidth}
\begin{tikzpicture}

\definecolor{darkgray176}{RGB}{176,176,176}

\begin{axis}[
width=1.0\fwidth,
height=1.0\fwidth,
point meta max=1, %
point meta min=0, %
tick align=outside,
tick pos=left,
x grid style={darkgray176},
xmin=-0.5, xmax=19.5,
xtick style={color=black},
y dir=reverse,
y grid style={darkgray176},
ymin=-0.5, ymax=19.5,
ytick style={color=black},
ylabel={index $i$},
xlabel={index $j$}
]
\addplot graphics [includegraphics cmd=\pgfimage,xmin=-0.5, xmax=19.5, ymin=19.5, ymax=-0.5] {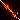};
\end{axis}

\end{tikzpicture}
\begin{tikzpicture}

\definecolor{darkgray176}{RGB}{176,176,176}

\begin{axis}[
width=1.0\fwidth,
height=1.0\fwidth,
colorbar,
colorbar style={ylabel={}},
colormap={mymap}{[1pt]
  rgb(0pt)=(0.0416,0,0);
  rgb(365pt)=(1,0,0);
  rgb(746pt)=(1,1,0);
  rgb(1000pt)=(1,1,1)
},
point meta max=1, %
point meta min=0, %
tick align=outside,
tick pos=left,
x grid style={darkgray176},
xmin=-0.5, xmax=19.5,
xtick style={color=black},
y dir=reverse,
y grid style={darkgray176},
ymin=-0.5, ymax=19.5,
ytick style={color=black},
xlabel={index $j$}
]
\addplot graphics [includegraphics cmd=\pgfimage,xmin=-0.5, xmax=19.5, ymin=19.5, ymax=-0.5] {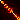};
\end{axis}

\end{tikzpicture}
\end{center}
\caption{Heatmap visualization of the squared dot products $|\langle v_i, v^{(\sigma)}_j \rangle|^2$ for the eigenvectors $v_i, i=1, ..., n$ ($y$-axis) of $A$ and $v_j$ of $A^{(\sigma)}$, $j=1, ..., n$ ($x$-axis).
Left plot refers to $\Omega = [0, 1]$, while right plot refers to $[0, 1]^3$. \\
It is easy to see that the largest values of $|\langle v_i, v^{(\sigma)}_j \rangle|^2$ are obtained close to the diagonal, i.e.\ for $i \approx j$.}
\label{fig:spectral_alignment}
\end{figure}

\section{Conclusion and outlook}
\label{sec:conclusion_outlook}

This paper refined a multivariate Ingham-type theorem,
which was then leveraged to prove novel stability estimates for kernel matrices.
In particular, a two-sided bound on the Rayleigh quotients of kernel matrices for finitely smooth kernels of different smoothness were derived.
These bounds allowed to understand a spectral alignment phenomenon between the eigenvectors of kernel matrices for kernels of different smoothness.

For future research,
we aim to apply these results to better understand the interplay of kernels of different smoothness, 
their corresponding RKHS as well as corresponding direct and inverse statements.
Furthermore, we want to extend the results to kernels of infinite smoothness and use these results for an improved understanding on infinitely smooth kernels.

\textbf{Acknowledgements:}
The authors acknowledge financial support through the projects LD-SODA of the {\em Landesforschungsf\"orderung Hamburg} (LFF)
and support from the RTG~2583 ``Modeling, Simulation and Optimization of Fluid Dynamic Applications''
funded by the {\em Deutsche Forschungsgemeinschaft} (DFG).
We thank Joshua Lampert for proofreading and the reviewers for valuable feedback.

\bibliography{/home/wenzel/references}				%
\bibliographystyle{abbrv}

\appendix

\appendix

\section{Proof of \Cref{th:ingham_upper_refined}}
\label{sec:proof_ingham}

To prove \Cref{th:ingham_upper_refined}, 
we follow the proof from \cite[Chapter 8]{komornik2005fourier}, 
however closely tracking the $R$-dependence and the constants involved.

\begin{proof}[Proof of \Cref{th:ingham_upper_refined}]
We consider $x_0 = 0$, 
as the general case $x_0 \neq 0$ is obtained by observing
\begin{align*}
\int_{B_R(x_0)} \left| \sum_{j=1}^{|X|} \alpha_j e^{i\omega^\top x_j} \right|^2 ~ \mathrm{d}\omega 
=& ~ \int_{B_R(0)} \left| \sum_{j=1}^{|X|} \alpha_j e^{i x_0^\top x_j } e^{i\omega^\top x_j} \right|^2 ~ \mathrm{d}\omega \\
\leq& ~ c_2 \cdot R^d \cdot \left \Vert (\alpha_j e^{ix_0^\top x_j})_{j=1}^{|X|} \right \Vert^2 \\
=& ~ c_2 \cdot R^d \cdot \Vert \alpha \Vert^2,
\end{align*}
and using the same calculation for the lower bound.

In the following, we occasionally write $B_R$ instead of $B_R(0)$. \\
For $x_0=0$, 
we start by proving the \underline{second inequality}:
\begin{itemize}
\item Let $H$ be the $\Vert \cdot \Vert_{L^2(B_{1/2}(0)}$ normalized eigenfunction of $-\Delta$ in $H_0^1(B_{\frac{1}{2}}(0)$ corresponding to the smallest eigenvalue.
As the Dirichlet Laplace operator $-\Delta$ is a positive operator, the eigenfunction $H$ can be chosen positive.
We extend $H$ to $\R^d$ by zero, still denoting it as $H$. 
It follows that $\tilde{H}(x) := R^{d/2} \cdot H(Rx)$ is the eigenfunction of $-\Delta$ in $H_0^1(B_{\frac{1}{2R}}(0)$ (likewise extended by zero to $\R^d$).
Let $h$ be the Fourier transform of $H$. 
Using the dilation property of the Fourier transform, we can calculate the Fourier transform $\tilde{h}$ of $\tilde{H}$ as
\begin{align}
\label{eq:h_positive}
\tilde{h}(\omega) = R^{-d/2} \cdot h(R^{-1} \omega).
\end{align}

For $\Vert \omega \Vert \leq \pi R$ and $\Vert x \Vert \leq \frac{1}{2R}$ we have $|\omega^\top x| \leq \frac{\pi}{2}$, i.e.\ $\cos(\omega^\top x) \geq 0$ and thus the integral is positive.
Thus we define
\begin{align}
\label{eq:definition_beta}
\beta := (2\pi)^{d/2} \cdot \left[ \min_{\omega \in B_\pi(0)} h(\omega) \right]^2 > 0.
\end{align}
\item We define $\tilde{G}_2 := \tilde{H} \ast \tilde{H} \in H_0^1(B_{\frac{1}{R}}(0))$ and obtain for its Fourier transform $\tilde{g}_2 = \mathcal{F}[\tilde{G}_2]$ via the convolution theorem
\begin{align*}
\tilde{g}_2 &= \mathcal{F}[\tilde{H} \ast \tilde{H}] = (2\pi)^{d/2} \cdot \mathcal{F}[\tilde{H}]^2 = (2\pi)^{d/2} \cdot \tilde{h}^2,
\end{align*}
i.e.\ $\tilde{g}_2$ is a non-negative function. 
We can express the minimal value of $\tilde{g}_2$ in $B_{\pi R}(0)$ with help of $\beta$:
\begin{align*}
\min_{\omega \in B_{\pi R}(0)} \tilde{g}_2(\omega) &= ~ (2\pi)^{d/2} \cdot \left[ \min_{\omega \in B_{\pi R}(0)} \tilde{h}(\omega) \right]^2 \\
&= ~ (2\pi)^{d/2} \cdot R^{-d} \left[ \min_{\omega \in B_{\pi}(0)} h(\omega) \right]^2 
= ~ R^{-d} \beta.
\end{align*}
\item With this we can finally compute for $q_{X} \geq \frac{1}{R}$:
\begin{align*}
R^{-d} \beta \cdot \int_{B_{\pi R}(0)}& \left| \sum_{j=1}^{|X|} \alpha_j e^{i\omega^\top x_j} \right|^2 ~ \mathrm{d}\omega 
\leq ~ \int_{B_{\pi R}(0)} \tilde{g}_2(\omega) \left| \sum_{j=1}^{|X|} \alpha_j e^{i\omega^\top x_j} \right|^2 ~ \mathrm{d}\omega \\
\leq& ~ \int_{\R^d} \tilde{g}_2(\omega) \left| \sum_{j=1}^{|X|} \alpha_j e^{i\omega^\top x_j} \right|^2 ~ \mathrm{d}\omega 
= ~ (2\pi)^{d/2} \cdot \sum_{i,j=1}^{|X|} \alpha_i \alpha_j \tilde{G}_2(x_i - x_j) \\
=& ~ (2\pi)^{d/2} \cdot \tilde{G}_2(0) \cdot \sum_{j=1}^{|X|} |\alpha_j|^2 
= ~ (2\pi)^{d/2} \cdot \tilde{G}_2(0) \cdot \Vert \alpha \Vert^2,
\end{align*}
which can be rearranged to
\begin{align}
\label{eq:almost_final_result}
\int_{B_{\pi R}(0)}& \left| \sum_{j=1}^{|X|} \alpha_j e^{i\omega^\top x_j} \right|^2 ~ \mathrm{d}\omega \leq (2\pi)^{d/2} \cdot
\frac{\tilde{G}_2(0)}{\beta} \cdot R^d \cdot \Vert \alpha \Vert^2.
\end{align}
The value of $\tilde{G}_2(0)$ is can be calculated explicitly by using the symmetry $H(x) = H(-x)$:
\begin{align}
\label{eq:Gnull_calculation}
\tilde{G}_2(0) &\equiv ~ \int_{\R^d} \tilde{H}(x) \tilde{H}(-x) ~ \mathrm{d}x 
= ~ \int_{B_{\frac{1}{2R}}} \tilde{H}(x) \tilde{H}(-x) ~ \mathrm{d}x \notag \\
&= ~ R^d \cdot \int_{B_{\frac{1}{2R}}} H(Rx) H(-Rx) ~ \mathrm{d}x
= ~ \int_{B_{\frac{1}{2}}} H(x) H(-x) ~ \mathrm{d}x \notag \\ %
&= ~ \int_{B_\frac{1}{2}} H(x)^2 ~ \mathrm{d}x = 1.
\end{align}
\item Replacing $\pi R$ by $R$ in Eq.~\eqref{eq:almost_final_result} implies also $q_X \geq \pi/R$, 
which then gives the final statement
\begin{align*}
\int_{B_{R}(0)}& \left| \sum_{j=1}^{|X|} \alpha_j e^{i\omega^\top x_j} \right|^2 ~ \mathrm{d}\omega \leq \frac{ 2^{d/2} }{ \pi^{d/2} \beta} \cdot R^d \cdot \Vert \alpha \Vert^2.
\end{align*}
\end{itemize}

\noindent We continue by proving the \underline{first inequality}:
\begin{itemize}
\item Let $H$ be again the eigenfunction of $-\Delta$ in $H_0^1(B_{\frac{1}{2}}(0)$ corresponding to the smallest eigenvalue $\lambda_{\min} := \lambda_{\min}(-\Delta) > 0$ and define again $\tilde{H}(x) := R^{d/2} \cdot H(Rx)$ with support in $B_{\frac{1}{2R}}(0)$.
This time we define 
\begin{align*}
G_1(x) &:= (r^2 + \Delta)(H \ast H)(x), ~
\tilde{G}_1(x) := (\tilde{r}^2 + \Delta)(\tilde{H} \ast \tilde{H})(x)
\end{align*}
for some $r, \tilde{r} > 0$ to be chosen.
It holds $\mathrm{supp}(G_1) \subseteq B_1(0)$ as well as $\mathrm{supp}(\tilde{G}_1) \subseteq B_{1/R}(0)$.
For the Fourier transform $g_1 := \mathcal{F}[G_1]$ and $\tilde{g}_1 := \mathcal{F}[\tilde{G}_1]$ we obtain by using standard properties
\begin{align*}
g_1(\omega) &= (r^2 - \Vert \omega \Vert^2) \cdot (2\pi)^{d/2} \cdot h(\omega)^2, \\
\tilde{g}_1(\omega) &= (\tilde{r}^2 - \Vert \omega \Vert^2) \cdot (2\pi)^{d/2} \cdot \tilde{h}(\omega)^2.
\end{align*}
From this we can conclude that $g_1(\omega) \leq 0$ for $\Vert \omega \Vert > r$ as well as $\tilde{g}_1(\omega) \leq 0$ for $\Vert \omega \Vert > \tilde{r}$.
From Eq.~\eqref{eq:h_positive} we recall $\tilde{h}(\omega) = R^{-d/2} \cdot h(R^{-1} \omega)$.

Furthermore, as $g_1$, $\tilde{g}_1$ are continuous, 
they are upper bounded on $B_r(0)$ respective $B_{\tilde{r}}(0)$, 
i.e.\ $\max_{\omega \in B_r(0)} g_1(\omega) < \infty$ and $\max_{\omega \in B_{\tilde{r}}(0)} \tilde{g}_1(\omega) < \infty$.

We can upper bound this maximum via
\begin{align*}
\max_{\omega \in B_r(0)} g_1(\omega) \leq r^2 \cdot (2\pi)^{d/2} \max_{\omega \in B_r(0)} h(\omega)^2, \\
\max_{\omega \in B_{\tilde{r}}(0)} \tilde{g}_1(\omega) \leq \tilde{r}^2 \cdot (2\pi)^{d/2} \max_{\omega \in B_{\tilde{r}}(0)} \tilde{h}(\omega)^2.
\end{align*}

Now we have for a point set $X$ with $q_X \geq 1/R$: 
\begin{align*}
&~(2\pi)^{d/2} \cdot \tilde{G}_1(0) \cdot \Vert \alpha \Vert^2 = (2\pi)^{d/2} \cdot \tilde{G}_1(0) \cdot \sum_{j=1}^{|X|} |\alpha_j|^2 \\
=&~ (2\pi)^{d/2} \cdot \sum_{i,j=1}^{|X|} \alpha_i \alpha_j \tilde{G}_1(x_i - x_j) 
= \int_{\R^d} \tilde{g}_1(\omega) \cdot \left| \sum_{j=1}^{|X|} \alpha_j e^{i \omega^\top x_j} \right|^2 ~ \mathrm{d}\omega \\
\leq&~ \int_{B_{\tilde{r}(0)}} \tilde{g}_1(\omega) \left| \sum_{j=1}^{|X|} \alpha_j e^{i \omega^\top x_j} \right|^2 \mathrm{d}\omega 
\leq \max_{\omega \in B_{\tilde{r}}(0)} \tilde{g}_1(\omega) \cdot \int_{B_{\tilde{r}}(0)} \left| \sum_{j=1}^{|X|} \alpha_j e^{i \omega^\top x_j} \right|^2 ~ \mathrm{d}\omega.
\end{align*}
This can be rearranged to
\begin{align}
\label{eq:intermediate_4}
\frac{(2\pi)^{d/2} \cdot \tilde{G}_1(0)}{\max_{\omega \in B_{\tilde{r}}(0)} \tilde{g}_1(\omega)} \cdot \Vert \alpha \Vert^2 \leq 
\int_{B_{\tilde{r}}(0)} \left| \sum_{j=1}^{|X|} \alpha_j e^{i \omega^\top x_j} \right|^2 ~ \mathrm{d}\omega.
\end{align}
For $\max_{\omega \in B_{\tilde{r}}(0)} \tilde{g}_1(\omega)$ we have due to Eq.~\eqref{eq:h_positive}
\begin{align*}
\max_{\omega \in B_{\tilde{r}}(0)} \tilde{g}_1(\omega) &\leq \tilde{r}^2 \cdot (2\pi)^{d/2} \max_{\omega \in B_{\tilde{r}}(0)} \tilde{h}(\omega)^2 \\
&= \tilde{r}^2 \cdot (2\pi)^{d/2} \max_{\omega \in B_{\tilde{r}}(0)} R^{-d} h(R^{-1}\omega)^2 \\
&= \tilde{r}^2 R^{-d} \cdot (2\pi)^{d/2} \max_{\omega \in B_{\tilde{r}/R}(0)} h(\omega)^2 \\
&\leq \tilde{r}^2 R^{-d} \cdot (2\pi)^{d/2} \max_{\omega \in \R^d} h(\omega)^2 \\
&\leq \tilde{r}^2 R^{-d} \cdot (2\pi)^{d/2} (2\pi)^{-d/2} \Vert H \Vert_{L^1(\R^d)}^2 \\
&\leq \tilde{r}^2 R^{-d} \cdot \sqrt{|B_{1/2}(0)|} \Vert H \Vert_{L^2(\R^d)}^2 \\
&\leq \tilde{r}^2 R^{-d}.
\end{align*}
To determine $\tilde{G}_1(0)$, we consider $\tilde{G}_1(x) \equiv (\tilde{r}^2 + \Delta)(\tilde{H} \ast \tilde{H})(x)$ and calculate for the first part
\begin{align*}
\tilde{r}^2 \cdot (\tilde{H} \ast \tilde{H})(x) &= \tilde{r}^2 \cdot \int_{\R^d} \tilde{H}(\tau)\tilde{H}(x-\tau) ~ \mathrm{d}\tau \\
&= \tilde{r}^2 R^d \cdot \int_{\R^d} H(R\tau) H(R(x-\tau)) ~ \mathrm{d}\tau \\
&= \tilde{r}^2 \cdot \int_{\R^d} H(s) H(Rx-s) ~ \mathrm{d}s \\
&= \tilde{r}^2 \cdot (H \ast H)(Rx) \\
\Rightarrow \quad \tilde{r}^2 \cdot (\tilde{H} \ast \tilde{H})(0) &= \tilde{r}^2 \cdot (H \ast H)(0).
\end{align*}
For the second part, we have by using Green's identity
\begin{align*}
\Delta (\tilde{H} \ast \tilde{H})(x) &= \int_{\R^d} \tilde{H}(\tau) \Delta_x \tilde{H}(x-\tau) ~ \mathrm{d}\tau \\
&= \int_{\R^d} \tilde{H}(\tau) \Delta_\tau \tilde{H}(x-\tau) ~ \mathrm{d}\tau \\
&= -\int_{\R^d} \nabla_\tau \tilde{H}(\tau) \nabla_\tau \tilde{H}(x-\tau) ~ \mathrm{d}\tau \\
&= -R^d \cdot \int_{\R^d} \nabla_\tau H(R\tau) \nabla_\tau H(R(x-\tau)) ~ \mathrm{d}\tau \\
&= - R^2 \cdot \int_{\R^d} \nabla_s H(s) \nabla_s H(Rx-s) ~ \mathrm{d}s.
\end{align*}
Due to $H(s) = H(-s)$ we conclude $(\nabla H)(-s) = -(\nabla H)(s)$ such that we have
\begin{align*}
\Delta (\tilde{H} \ast \tilde{H})(0) &= - R^2  \int_{\R^d} \nabla_s H(s) \nabla_s H(-s) ~ \mathrm{d}s \\
&= R^2  \int_{\R^d} (\nabla H)(s) (\nabla H)(-s) ~ \mathrm{d}s \\
&= -R^2  \int_{\R^d} \nabla_s H(s) \nabla_s H(s) ~ \mathrm{d}s \\
&= -R^2 \lambda_{\min} \cdot \int_{\R^d} H(s)H(s) ~ \mathrm{d}s \\
&= -R^2 \lambda_{\min} \cdot \int_{\R^d} H(s)H(-s) ~ \mathrm{d}s \\
&= -R^2 \lambda_{\min} \cdot (H \ast H)(0) \\
&= -R^2 \lambda_{\min}.
\end{align*}

Overall we have $\tilde{G}_1(0) = (\tilde{r}^2 - R^2\lambda_{\min})$.
Thus we can continue with Eq.~\eqref{eq:intermediate_4} to obtain
\begin{align*}
(2\pi)^{d/2} \cdot \frac{(\tilde{r}^2 - R^2\lambda_{\min})}{\tilde{r}^2} \cdot R^d \Vert \alpha \Vert^2 \leq 
\int_{B_{\tilde{r}}(0)} \left| \sum_{j=1}^{|X|} \alpha_j e^{i \omega^\top x_j} \right|^2 ~ \mathrm{d}\omega.
\end{align*}
We pick $\tilde{r}^2 := 2 R^2 \lambda_{\min}$, such that the factor on the left of the previous line simplifies to $1/2$ and we have
\begin{align*}
\frac{(2\pi)^{d/2}}{2} \cdot R^d \Vert \alpha \Vert^2 \leq 
\int_{B_{\sqrt{2}R \lambda_{\min}^{1/2}}(0)} \left| \sum_{j=1}^{|X|} \alpha_j e^{i \omega^\top x_j} \right|^2 ~ \mathrm{d}\omega.
\end{align*}
Replacing $\sqrt{2}R \lambda_{\min}^{1/2}(-\Delta)$ by $R$ gives
\begin{align*}
\frac{(2\pi)^{d/2}}{2^{1+d/2} \lambda_{\min}^{d/2}} \cdot R^d \Vert \alpha \Vert^2 \leq 
\int_{B_R(0)} \left| \sum_{j=1}^{|X|} \alpha_j e^{i \omega^\top x_j} \right|^2 ~ \mathrm{d}\omega
\end{align*}
for any point set $X$ with $q_X \geq \sqrt{2}\lambda_{\min}^{1/2}/R$.
\end{itemize}
This concludes the proof of the statement with the constants
\begin{align}
\begin{aligned}
\label{eq:constants_c1_c2}
c_1 &:= \frac{\pi^{d/2}}{2 \lambda_{\min}^{d/2}(-\Delta)}, ~~
c_2 := \left(\frac{2}{\pi}\right)^{d/2} \cdot \frac{1}{\beta}
\end{aligned}
\end{align}
only depending on $d \in \N$.

For the second inequality we have the condition $q_X \geq \pi/R \Leftrightarrow R \geq \pi / q_X$, 
while the first inequality requires the condition $q_X \geq \sqrt{2}\lambda_{\min}^{1/2}(-\Delta)/R \Leftrightarrow R \geq \sqrt{2}\lambda_{\min}^{1/2}(-\Delta) / q_X$.
Note that $\sqrt{2} \lambda_{\min}^{1/2}(-\Delta) > \pi$ for any $d \in \N$.
\end{proof}

\end{document}